\documentclass[reqno]{amsart}
\usepackage{relsize}
\usepackage{graphicx}
\usepackage{color}
\newcommand{\eps}{{\varepsilon}}
\renewcommand{\d}{\mathrm{d}}
\newcommand{\D}{\mathrm{D}}
\newcommand{\e}{\mathrm{e}}
\newcommand{\R}{\mathbb{R}}
\newcommand{\N}{\mathbb{N}}

\newcommand{\A}{\mathcal{A}}
\newcommand{\F}{\mathcal{F}}
\newcommand{\G}{\mathcal{G}}
\newcommand{\Rho}{\mathcal{R}}

\newtheorem{theorem}{Theorem}
\newtheorem{lemma}[theorem]{Lemma}

\newtheorem{corollary}[theorem]{Corollary}
\theoremstyle{remark}
\newtheorem{remark}{Remark}

\graphicspath{{figures/}}

\begin{document}
\title[Optimal balance via adiabatic invariance]%
{Optimal balance via adiabatic invariance of approximate slow
manifolds}
\date{\today}

\author[G.A. Gottwald]{Georg A. Gottwald}
\address[G.A. Gottwald]%
{School of Mathematics and Statistics \\
 University of Sydney \\
 NSW 2006 \\
 Australia}

\author[H. Mohamad]{Haidar Mohamad}
\address[H. Mohamad]%
{School of Engineering and Science \\
 Jacobs University \\
 28759 Bremen \\
 Germany}

\author[M. Oliver]{Marcel Oliver}
\address[M. Oliver]%
{School of Engineering and Science \\
 Jacobs University \\
 28759 Bremen \\
 Germany}

%\subjclass[2010]{34E13, 70H03, 70H09, 70H15, 78A35}
%\date{June 12, 2014}

\begin{abstract}
We analyze the method of optimal balance which was introduced by
Vi{\'u}dez and Dritschel (J. Fluid Mech.\ \textbf{521}, 2004, pp.\
343--352) to provide balanced initializations for two-dimensional and
three-dimensional geophysical flows, here in the simpler context of a
finite dimensional Hamiltonian two-scale system with strong gyroscopic
forces.  It is well known that when the potential is analytic, such
systems have an approximate slow manifold that is defined up to terms
that are exponentially small with respect to the scale separation
parameter.  The method of optimal balance relies on the observation
that the approximate slow manifold remains an adiabatic invariant
under slow deformations of the nonlinear interactions.  The method is
formulated as a boundary value problem for a homotopic deformation of
the system from a linear regime, where the slow-fast splitting is
known exactly, to the full nonlinear regime.  We show that, providing
the ramp function which defines the homotopy is of Gevrey class $2$
and satisfies vanishing conditions to all orders at the temporal end
points, the solution of the optimal balance boundary value problem
yields a point on the approximate slow manifold that is exponentially
close to the approximation to the slow manifold via exponential
asymptotics, albeit with a smaller power of the small parameter in the
exponent.  In general, the order of accuracy of optimal balance is
limited by the order of vanishing derivatives of the ramp function at
the temporal end points.  We also give a numerical demonstration of
the efficacy of optimal balance, showing the dependence of accuracy on
the ramp time and the ramp function.
\end{abstract}

\maketitle

\section{Introduction}

Nonlinear Hamiltonian two-scale systems with a single fast frequency
possess an approximate slow manifold: a region in phase space
characterized by smallness of an adiabatically invariant ``fast
energy''.  A trajectory near the approximate slow manifold will stay
near it for a long period of time---often exponentially long with
respect to the scale separation parameter under suitable assumptions
(see, e.g., \cite{MacKay04,Neishtadt81}).  It is important to stress
that, despite the language used, this phase space region is not a
manifold in any rigorous sense (except in trivial cases such as linear
ODEs).  Rather, it is described by a generally diverging asymptotic
series \cite{Van13}.

An explicit description of an approximate slow manifold is usually
only practical to a low fixed order of asymptotics because the number
of terms grows exponentially with order.  Optimal truncation, a
powerful theoretical tool e.g.\ for proving almost-invariance over
exponentially long times, cannot be implemented in a computational
model.  It is, however, possible to numerically compute single points
on the approximate manifold with an accuracy that is nearly as good as
optimal truncation.  This procedure, which we refer to as
\emph{optimal balance}, is the subject of this paper.

The idea underlying optimal balance is that adiabatic invariants of
the unperturbed dynamics remain adiabatic under slowly varying
perturbations.  If a homotopy varying in slow time perturbs the system
from linear to fully nonlinear, trajectories that emerge from the
known slow subspace at the linear end will connect to an approximately
slow fully nonlinear state at the other end.  Computationally, this
amounts to solving a boundary value problem where the boundary
condition at the linear end constrains to the slow linear subspace and
the boundary condition at the fully nonlinear end constrains to the
slow base-point coordinate of the approximate manifold.

Our motivation comes from studying balance in geophysical fluid flow.
On large scales in the mid-latitudes, the Coriolis force nearly
balances the pressure gradient force while inertial forces are
subdominant.  As a result, the flow approximately splits into a slow
\emph{balanced} component which evolves nonlinearly and interacts only
weakly with the fast components which are approximately described by
linear waves.  A precise characterization of this splitting is a
perennial theme in geophysical fluid dynamics; we refer the reader to
the reviews of Vanneste \cite{Van13} and McIntyre \cite{McI15} for a
more comprehensive background.

A computational procedure for describing balance is of considerable
practical importance.  First, unphysically imbalanced initial
conditions may require unnecessary large amounts of artificial
viscocity to ensure stability in a numerial simulation; thus, accurate
balancing can improve numerical accuracy, in particular when
frontogenesis is important \cite{cullen2008comparison}.  Second, in
studies of the role of inertial-gravity waves in the energy budget of
the ocean, accurate diagnostics are currently lacking; optimal balance
may provide a way to diagnose small imbalanced components in large
unsteady flows with minimal ambiguity
\cite{VonStorchBO:2017:InteriorEP}.  Third, enforcing balance is a
practical necessity when assimilating noisy observations to initialize
a weather forecast; failing to do so results in spectacular failure
(see, e.g., the wonderful historical account in \cite{Lynch}).  To
assure that the assimilated state is consistently balanced, the
analysis output is typically post-processed, e.g.\ using a digital
filter \cite{LynchHuang92}.  Dynamical information about imbalance and
approximate slow manifolds has only recently become part of the actual
data assimilation procedure \cite{Neef06,Gottwald14}.  Cotter
\cite{Cotter13}, in particular, demonstrates that optimal balance can
be used as a constraint when assimilating balanced states in a simple
two-scale Hamiltonian model problem.

The method of optimal balance for rotating fluid flow was first
proposed by Vi{\'u}dez and Dritschel \cite{ViudezDritschel04}.  In
their work, they coin the term ``optimal potential vorticity balance''
which reflects that rather than deforming the equations of motion,
they ramp up the vorticity anomaly in the initial data.
Mathematically, this is equivalent to homotopically turning on
nonlinear interactions.  In practical terms, this is only feasible
when using a potential-vorticity-based fully Lagrangian code.  In
their work, they suggest a simple iterative scheme to solve the
resulting boundary value problem and report good behavior both in
terms of convergence of the algorithm and in terms of quality of
balance as measured by independent diagnostics.

Cotter \cite{Cotter13} studies optimal balance for data assimilation
using a simple finite-dimensional Hamiltonian system which has been
used as a prototype model for balance in a number of previous studies
\cite{Oliver06,CotterReich06,GottwaldO:2014:SlowDD}.  In particular,
Cotter points out that their earlier results \cite{CotterReich06}
imply rigorous exponential estimates for analytic ramp functions with
exponentially decaying tails.

In the present paper, we consider optimal balance in the same
finite-dimensional setting on a fixed finite interval in slow time.
In this setting, the asymptotic behavior of the method is not only
determined by the smoothness of the potential and of the ramp
function, but also by the order of vanishing of the derivatives of the
ramp function at the temporal end points.  When the derivatives of the
ramp function vanish only up to some finite order $k$ at the initial
and at the final time of the ramp, the rate of convergence of optimal
balance is limited to $O(\eps^{k+1})$, where $\eps$ denotes the
time-scale separation parameter.  Correspondingly, beyond-all-order
accuracy requires that \emph{all} derivatives of the ramp function
vanish at the end points.  However, the ramp function cannot be
simultaneously uniformly analytic and satisfy the correct end-point
conditions.  Here, we show that exponential estimates can still be
achieved when the ramp function is not analytic but of Gevrey class
$2$.

When the potential in this model is analytic, classical Hamiltonian
normal form theory states that there exists a constant $c$ and a
symplectic transformation which approximately splits the system into
fast and slow variables such that when the fast variable is initially
zero, it remains $O(\exp(-c/\eps))$ over times of $O(\exp(c/\eps))$ as
$\eps \to 0$ \cite{CotterReich06}.  Here we prove that optimal
balance, for ramp functions described above, yields a state that, if
used as initial data for the original system, corresponds to a
normal-form fast variable that remains $O(\exp(-c/\eps^{1/3}))$-small
over times of $O(\exp(c/\eps))$.  We present numerical results that
indicate that the exponent $\tfrac13$ is not sharp, but that an
exponent $1$ as in the classical normal form result cannot be
achieved.

This result provides a strong justification of the method of optimal
balance: the algorithm yields a point on the approximate slow manifold
that is exponentially close to what could be obtained from an
optimally truncated asymptotic expansion.

The method of proof has a long history.  A concise mathematical
treatment of adiabatic invariance for linear systems is given by Leung
and Meyer \cite{LeungMeyer75}, we refer the reader to this paper for
some of the early history.  Exponential estimates for nonlinear
systems are due to Nekhoroshev \cite{Nekhoroshev77} and Ne{\u\i}shtadt
\cite{Neishtadt81,Neishtadt84}.  Cotter and Reich \cite{CotterReich06}
apply this theory to the model problem under consideration here.  In
our proof, we do not use their Hamiltonian setting because the
recursive step is only easy when applying a Cauchy estimate at each
iteration.  When the ramp function is not analytic but only Gevrey
class $2$, Cauchy estimates are not available and the iteration does
not directly close up.  As we do not need estimates over times longer
than $O(1)$ in slow time, we resort to more direct estimates on an
explicit construction of the fast-slow splitting as used in
\cite{GottwaldO:2014:SlowDD}.

The paper is organized as follows. In Section~\ref{s.model}, we detail
the finite-dimensional model for balance and review the direct
construction of the slow vector field.  In Section~\ref{s.ramp}, we
describe the method of optimal balance applied to this model.  We
state and prove our main theorems on optimal balance in
Sections~\ref{s.averaging} and \ref{s.analytic} for $C^k$ potentials
and for analytic potentials, respectively.  Section~\ref{s.numerics}
presents numerical simulations corroborating our analytical results.
Section~\ref{s.summary} concludes with a discussion and open
questions.

\section{The model}
\label{s.model}

We consider the Hamiltonian system of differential equations
\begin{subequations}
  \label{e.hamiltonian}
\begin{align}
  \dot q & = p \,, \\
  \dot p & = J p - \eps\, \nabla V(q) \,,
\end{align}
\end{subequations}
where $q \colon [0,T] \to \R^{2d}$ is the vector of positions, $p$ the
vector of corresponding momenta, $J$ is the canonical symplectic
matrix in $2d$ dimensions, $V$ is a smooth potential and $\eps$ is a
small parameter.

When $d=1$, this system can be interpreted as describing the motion of
a single Lagrangian particle in the rotating shallow water equations
with frozen height field
\cite{Oliver06,CotterReich06,FrankGottwald13}.  In this
interpretation, $J \dot q$ represents the Coriolis force and
$\eps \to 0$ describes the limit of rapid rotation.  Alternatively,
\eqref{e.hamiltonian} can be seen as describing the motion of a single
charged particle in a planar potential $V$ under the influence of a
magnetic field normal to the plane of motion. In this interpretation,
$J \dot q$ represents the Lorentz force and $\eps \to 0$ corresponds
to the mass of the particle going to zero while its charge remains
constant.

The system \eqref{e.hamiltonian} is Hamiltonian, albeit with a
non-canonical symplectic structure.  To leading order, the splitting
into slow and fast degrees of freedom can be determined by inspection.
Indeed, rescaling to slow time $\tau = \eps t$, introducing a slow
momentum $\pi = p/\eps$, and setting $\eps=0$, we see that the leading
order slow dynamics is given by
\begin{equation}
  \frac{\d q}{\d \tau} = -J \nabla V(q) \,,
\end{equation}
so that the corresponding leading order fast variable is
$\omega = \pi + J \nabla V(q)$.  This splitting can be iteratively
refined by adding higher order terms, which gives an explicit formula
for the $n$th-order slow vector field $G_n (q)$ which is needed as a
reference for the optimal balance vector field to compare against and
which is stated here in terms of the original fast time variables.

\begin{theorem}[\cite{GottwaldO:2014:SlowDD}]
\label{t.nonvariational} 
For $n \in \N$, suppose $V \in C^{n+2}$ and set
\begin{equation}
  G_n (q) = \eps \sum_{i=0}^n g_i(q) \, \eps^{i} 
  \label{e.Fnvdef}
\end{equation}
with coefficient functions $g_i$ recursively defined via
\begin{subequations}
  \label{e.bks-orig}
\begin{align}
  g_0 (q) & = -J \nabla V(q) \,, \\
  g_k (q) & = -J \sum_{i+j=k-1} \D g_i(q) \, g_j(q) \,.
\end{align}
\end{subequations}
For fixed $q_0 \in \R^{2d}$ and $a>0$, let $q(t)$ denote a solution to
\begin{equation}
  \dot q =G_n(q)
  \label{e.Fnvtot}
\end{equation}
with $q(0)=q_0$.  Let $q_\eps(t)$ solve the full parent dynamics
\eqref{e.hamiltonian} consistently initialized via $q_\eps(0)=q_0$ and
$p_\eps(0) = G_n(q_0)$.  Then there exists
$\eps_0>0$ and $c=c(q_0, a, V)$ such that
\begin{equation}
  \sup_{t\in[0,a/\eps]} \lVert q_\eps(t) - q(t) \rVert 
  \leq c \, \eps^{n+2} 
  \label{e.nonvariational}
\end{equation}
for all $0 < \eps \leq \eps_0$.  
\end{theorem}

We note that this result does not fully use the Hamiltonian structure;
it only relies on the anti-symmetry of $J$.  Thus, the resulting
estimate is valid only over slow times of order one.  Hamiltonian
normal form theory will yield estimates that remain valid for much
longer times \cite{CotterReich06}.  For our purposes this is not
required, but we make use of the explicit form of the slow vector
field.

\section{Optimal balance}
\label{s.ramp}

On a conceptual level, optimal balance works by homotopically deforming
the system into a simpler, e.g.\ linear system where the slow manifold
is trivial to compute.  Figure~\ref{f.sketch} provides a sketch where
the slow manifold at $t=0$ is described by $p=0$.  The homotopy
generates a surface of approximate slow manifolds in the extended
phase space, illustrated by the green shaded surface.  In general, for
a frozen value of the homotopy parameter, the ``manifold''
$\mathcal{M}$ is only approximately invariant: trajectories drift away
exponentially slowly with respect to the scale separation parameter
$\eps$.  This is indicated by the dotted green line.

\begin{figure}[tb]
\centering
\includegraphics[width=\textwidth]{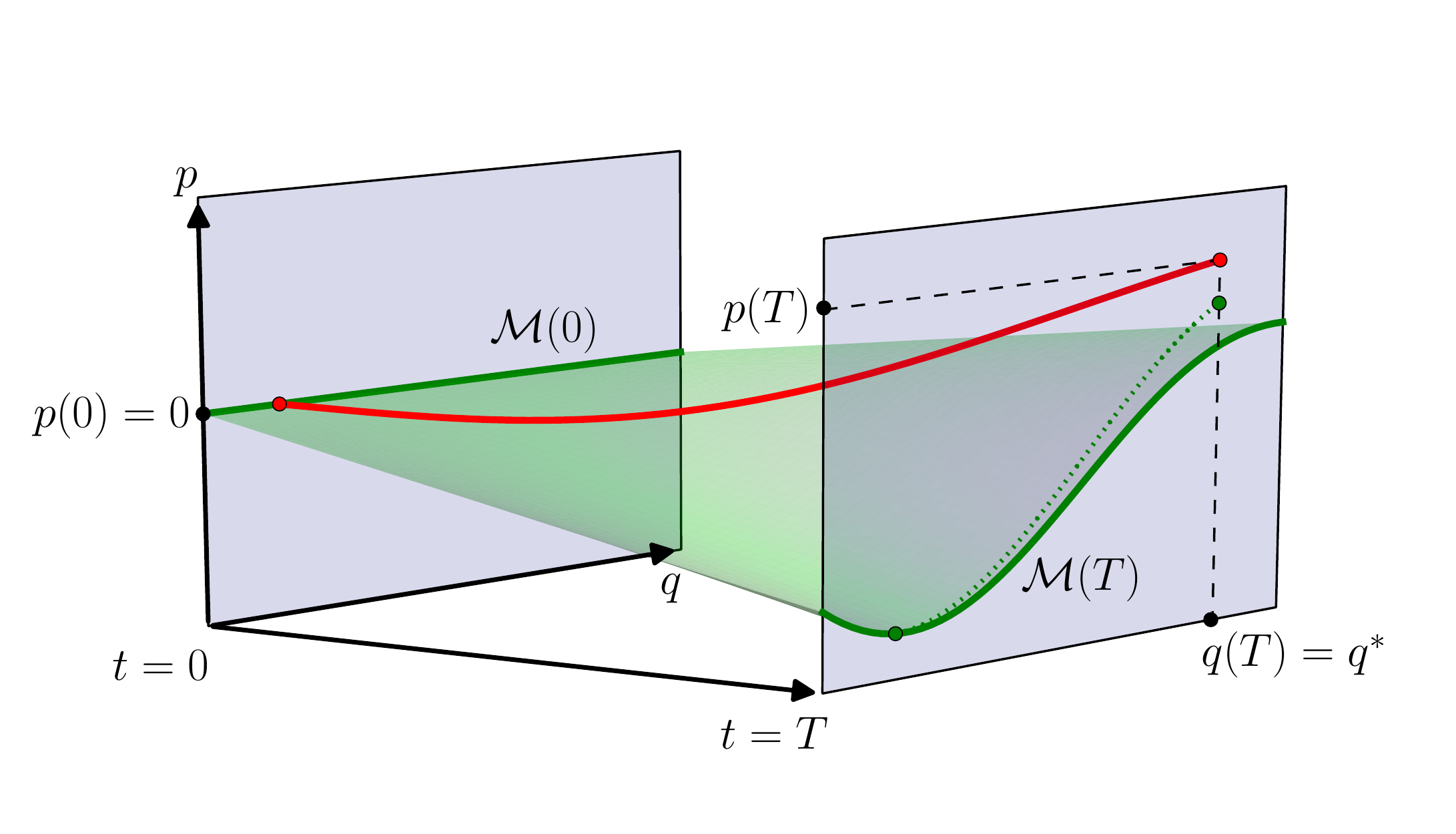}
\caption{Sketch of the geometry of optimal balance in extended phase
space.}
\label{f.sketch}
\end{figure}

In the optimal balance, we identify the homotopy parameter with slowly
varying time.  In this case, the approximate slow manifold is an
adiabatic invariant: a trajectory initially on the slow manifold will
stay near it for very long time while the manifold deforms slowly.  Such a
trajectory is shown in red in Figure~\ref{f.sketch}. In this case,
there are two sources of drift: on the one hand the drift already
present for a frozen homotopy parameter.  On the other hand, the drift
due to the deformation of the manifold in time.  In the following, we
shall estimate both sources of drift.

Our task is to specify a single point on the approximate slow manifold
$\mathcal{M}(T)$ by computing the fiber coordinate $p^*$ for a given
base-point coordinate $q^*$.  In the extended phase space picture of
Figure~\ref{f.sketch}, this corresponds to specifying two boundary
conditions: $q(T)=q^*$ and $p(0)=0$.  The first condition fixes the
base-point.  The second condition ensures that the entire trajectory
remains near $\mathcal{M}(t)$ for all $t \in [0,T]$.  We then define
$p^*=p(T)$ as the computational approximation of the fiber coordinate.

For the prototype model \eqref{e.hamiltonian}, the procedure can be
stated as follows.  Take a smooth monotonic ramp function
$\rho \colon [0,1] \to [0,1]$ with $\rho(0)=0$ and $\rho(1)=1$.  For
given $q^* \in \R^{2d}$, fix a ramp time $T>0$ and solve the boundary
value problem
\begin{subequations}
  \label{e.ramp}
\begin{align}
  \dot q & = p \,, \\
  \dot p & = J p - \eps \, \rho(t/T) \, \nabla V(q) \,,
\end{align}
with boundary conditions 
\begin{equation}
  p(0)=0 \qquad \text{and} \qquad q(T)=q^* \,.
\end{equation}
\end{subequations}
Then set $p^*=p(T)$.

We note that when the ramp parameter is frozen at $t=0$,
\eqref{e.ramp} reduces to the trivial linear fast-slow system
$\dot q = p$ and $\dot p = Jp$, where $p$ is fast and $q$ is slow.
This justifies the initial-time boundary condition $p(0)=0$.  We note
that the boundary value $q(0)$ is not used explicitly in this setup.  

In the following two sections, we analyze the accuracy of optimal
balance by comparing against the slow vector field $G_n$ associated
with the original dynamical system \eqref{e.hamiltonian}, given by
Theorem~\ref{t.nonvariational}.  We shall see that the asymptotic
construction of the slow vector field for the ramped system
\eqref{e.ramp} contains additional terms at $O(\eps^{k+1})$ unless all
derivatives of $\rho$ up to order $k$ vanish at the final time.
Similarly, the description of the trivial slow manifold $p=0$ differs
from the description of the slow manifold for the ramped system
\eqref{e.ramp} at $O(\eps^{k+1})$ unless all derivatives of $\rho$ up
to order $k$ vanish at the intial time.  Thus, the order of accuracy
of optimal balance is limited by the order of vanishing of derivatives
of the ramp function at the end points.

\section{Algebraic optimal balance}
\label{s.averaging}

In this section, we consider the case when $V$ or the ramp function
$\rho$ are only finitely differentiable.  Then the best we can expect
is an algebraic rate of convergence of optimal balance.  The proof is
a straightforward generalization of the classical fast-slow splitting
used to prove Theorem~\ref{t.nonvariational} in
\cite{GottwaldO:2014:SlowDD}.

\begin{theorem} \label{mainTheorem}
For $n \in \N$, suppose $\rho \in C^{n+1}([0,1])$ with $\rho(0)=0$ and
$\rho(1)=1$ satisfying the algebraic order condition
\begin{gather}
  \rho^{(i)}(0) = \rho^{(i)}(1) = 0
\end{gather}
for $i = 1, \dots, n$.  Suppose further that $V \in C^{n+2}$.  Fix
$a>0$ and consider a sequence of ramp times $T=a/\eps$ and a sequence
of solutions $(q,p)$, implicitly parameterized by $\eps$, to the
boundary value problem \eqref{e.ramp}.  Then there exists a constant
$c=c(\rho, a, n, V)$ such that
\begin{gather}
  \lVert p(T) -  G_n(q^\ast) \rVert 
  \le c \, \varepsilon^{n+2} \,.
  \label{e.main}
\end{gather}
\end{theorem}

\begin{proof}
By choosing appropriate units of time, we can take $a=1$ without loss
of generality.  We then introduce the fast variable
$w = p - \eps \, F_{n+1}$, where
\begin{equation}
  F_n (q,t) = \sum_{i=0}^n f_i(q,t) \, \eps^i 
  \label{e.Fnvdef}
\end{equation}
with coefficients $f_i$ to be determined.  Then
\begin{subequations}
\begin{align}
  \dot q 
  & = \eps \, F_{n+1} + w \,, 
      \label{e.transformed.a} \\
  \dot w
  & = (J - \eps \, \D F_{n+1}) \, w
      + \eps \, (J F_{n+1} - \rho \, \nabla V)
      - \eps^2 \, \partial_\tau F_{n+1} 
      - \eps^2 \, \D F_{n+1} \, F_{n+1} 
  \label{e.transformed.b} 
\end{align}
\end{subequations}
where, as before, $\tau = \eps t$ so that
$\partial_t = \eps \partial_\tau$.  We now eliminate the inhomogeneous
term on the right of \eqref{e.transformed.b} order by order up to an
$O(\eps^{n+2})$ remainder.  This leads to the recursive expression
\begin{subequations}
\begin{align}
  f_0 & = -\rho \, J \nabla V(q) \,, \\
  f_k & = - J \, \partial_\tau f_{k-1} 
          - J \sum_{i+j=k-1} \D f_i(q) \, f_j(q) 
\end{align}
  \label{e.bks}
\end{subequations}
for $k = 1, \dots, n+1$.  We remark that for $\rho \equiv 1$, $F_n$
reduces to $G_n$ introduced in \eqref{e.Fnvdef}.
Thus,
\begin{equation}
  \label{wRemainder}
  \dot w = (J - \eps \, \D F_{n+1}) \, w + O(\eps^{n+2}) 
\end{equation}
so that left-multiplying with $w$ implies
\begin{equation}
  \label{wRemainder1}
  \frac{\d}{\d t} \lVert w \rVert 
  \leq \eps \, \lVert \D F_{n+1} \rVert \, \lVert w \rVert 
       + O(\eps^{n+2}) \,. 
\end{equation} 
By assumption, $\partial_\tau^i \rho(0) = \rho^{(i)}(0) = 0$ so that
$f_i(q,0)=0$ for $i=1, \dots, n$.  Recalling that $p(0)=0$, we obtain
\begin{equation}
  w(0) = p(0) - \eps \, F_n(q(0),0) - \eps^{n+2} \, f_{n+1}(q(0),0)
       = O(\eps^{n+2}) \,.
\end{equation}
Hence, applying the Gronwall lemma to \eqref{wRemainder1}, we find
that there exists $c = c(\rho,n,V)$ such that
\begin{equation}
  \sup_{t \leq T} \, \lVert p(t) - \eps \, F_n(q(t),t) \rVert
  \leq c \, \eps^{n+2} \,.
\end{equation}
Comparing with \eqref{e.bks-orig} and noting that $\rho^{(i)}(1)=0$
for $ i= 1, \dots, n$, we see that
\begin{gather}
  G_n (q^\ast) = \eps \, F_n(q(T),T) \,.
\end{gather}
The required estimate \eqref{e.main} follows.
\end{proof}

\begin{corollary}
Suppose that, in the setting of Theorem~\ref{mainTheorem}, $V$ is
analytic and asymptotically strictly convex.  Then for every desired
order $n\in \N$ there exists a ramp function so that the method of
optimal balance generates a state which remains balanced to
$O(\eps^{n+2})$ over times of $O(\exp(c/\eps))$ under the dynamics of
system \eqref{e.hamiltonian}.
\end{corollary}

This result is a consequence of the uniqueness of the asymptotic
expansion.  More specifically, the fast variable
$p-G_n$ in our construction and the fast
variable $p_\eps$ in the Hamiltonian normal form setting of
\cite{CotterReich06} coincide up to terms of $O(\eps^{n+2})$.  Thus,
optimal balance at $O(\eps^{n+2})$ in the sense of
Theorem~\ref{mainTheorem} is equivalent to $p_\eps = O(\eps^{n+2})$ in
the notation of \cite{CotterReich06}.  Since $V$ is assumed analytic
and asymptotic strict convexity of $V$ implies that trajectories
remain in a compact subset of phase space for all times,
\cite[Theorem~2.1]{CotterReich06} applies and yields persistent
$O(\eps^{n+2})$ smallness of the fast variable over exponentially long
times.

\section{Exponential optimal balance}
\label{s.analytic}

In this section, we refine the result of Section~\ref{s.ramp} for the
case when $V$ is analytic and $\rho$ is in Gevrey class $2$.  

Let us first recall that a function $f \in C^\infty(U)$ for
$U \subset \R$ open is of Gevrey class $s$ if there exist constants
$C$ and $\beta$ such that
\begin{equation}
  \sup_{x \in U} \lvert f^{(n)}(x) \rvert
  \leq C \, \frac{n!^s}{\beta^n}
  \label{e.gevrey}
\end{equation}
for all $n \in \N$; see, e.g., \cite{ito1993encyclopedic}.  We write
$f \in G^s(U)$.  Then the following is true.

\begin{theorem}\label{mainTheorem2}
Suppose $\rho \in G^2 (0,1)$ with $\rho(0)=0$ and $\rho(1)=1$
satisfying the exponential order condition
\begin{equation}
  \label{CRho}
  \rho^{(i)}(0) = \rho^{(i)}(1) = 0 
\end{equation}
for all $i \in \N^*$.  Fix $a>0$ and consider a sequence of ramp times
$T=a/\eps$ and a sequence of solutions $(q,p)$, implicitly
parameterized by $\eps \leq 1$, to the boundary value problem
\eqref{e.ramp}.  Now suppose there exists a compact subset of phase
space $\mathcal{K}\subset \R^{2d}$ containing this sequence of
solution trajectories and that there exist $R>0$ and $z_0 \in \R^{2d}$
with $\mathcal{K}\subset B_{R/2}(z_0)$ such that $V$ is analytic on
$B_R(z_0)$.  Then there exist $n = n(\rho, a, V, \eps) \in \N$ and
positive constants $c = c(\rho, a, V)$ and $d = d(\rho, a, V)$ such
that
\begin{gather}\label{MainEstimExp}
  \lVert p(T) -  G_n(q^\ast) \rVert 
  \leq d \, \e^{-c \eps^{-\frac{1}{3}}} \,.
\end{gather}
\end{theorem}

To prove this theorem, we proceed as in the proof of
Theorem~\ref{mainTheorem}, albeit with a more careful estimate on the
remainder term.  Defining $w$ as before, we write equation
\eqref{wRemainder} in the form
\begin{equation}
  \dot{w} = (J -\eps \, \D F_{n+1}) \, w - R_{n+1} \,, 
\end{equation}
with explicit remainder 
\begin{gather}
  \label{e.remainder}
  R_{n+1} 
  = \eps^{n+3} \, \partial_\tau f_{n+1} +
    \sum_{k=n+1}^{2(n+1)} \eps^{k+2} 
    \sum_{\substack{i+j=k \\ i,j \leq n+1}} \D f_i \, f_j \,.
\end{gather}

The key observation is that each of the terms appearing in the
expression for $f_k$, and each of the terms appearing in the
expression for the remainders $R_k$, is a product of functions which
only depend on $\rho$ with functions which only depend on $V$.  Hence,
they can be written as inner products of coefficient vectors encoding
all $\rho$-dependence with coefficient vectors encoding all
$V$-dependence.  A H\"older-like inequality will separate the two, so
that we can estimate each class of coefficients separately in their
respective norms.

To formalize this idea, we need to introduce some notation.  We define
the Cartesian product $\F \oplus \G$ of two vectors
$\F = (\F^1, \dots, \F^N)$ and $\G = (\G^1, \dots, \G^M)$ as
\begin{equation}
  \F \oplus \G = (\F^1, \dots, \F^N, \G^1, \dots, \G^M)
\end{equation}
and the tensor product $\mathcal{A} \otimes \G$ of a vector of linear
operators $\mathcal{A} = (\A^1, \dots, \A^N)$ acting on a vector
$\G = (\G^1, \dots, \G^M)$ as
\begin{equation}
  \mathcal{A} \otimes \G
  = (\mathcal{A}^1 \G^1, \dots, \mathcal{A}^1 \G^M, \dots,
     \mathcal{A}^N \G^1, \dots, \mathcal{A}^N \G^M) \,.
\end{equation}
Further, we define the vector family $\{\Rho_k\}$ as
\begin{subequations}
\begin{gather}
  \Rho_0 = \rho \,, \\
  \Rho_{k+1} = \partial_\tau \Rho_k \oplus
  \bigoplus_{\substack{i+j=k }}  \Rho_i \otimes \Rho_j 
  \quad \text{for} \quad
  k = 0, \dots, n \,, \\  
  \Rho_{k+1} = \bigoplus_{\substack{i+j=k \\ i,j \leq n+1}} 
  \Rho_i \otimes \Rho_j
  \quad \text{for} \quad k = n+1, \dots, 2n+2
\end{gather}
\end{subequations}
where the components of $\Rho_j$ are acting on the components of
$\Rho_l$ by multiplication and the indexed Cartesian product can be
performed in any order so long as the order convention remains fixed
throughout, and the family $\{\F_k\}$ as
\begin{subequations}
  \label{e.Fk}
\begin{gather}
  \F_0 = -J\nabla V \,, \\
  \F_{k+1} 
  = - \biggl( 
        \F_k \oplus \bigoplus_{j+l=k}\D\F_j\otimes \F_l 
      \biggr) \, \mathbb{J}_{k+1}
  \quad \text{for} \quad
  k = 0, \dots, n+1 \,, \\
  \F_{k+1}
  = - \biggl(
        \bigoplus_{\substack{i+j=k \\ i,j \leq n+1}}
        \D\F_j \otimes \F_l \biggr) \, \mathbb{J}_{k+1}
  \quad \text{for} \quad
  k = n+2, \dots, 2n+2 \,,
\end{gather}
\end{subequations}
where $\mathbb{J}_{k+1}$ denotes the block-diagonal matrix of matching
dimension with $J$ on the main diagonal.

As the recursive structure of the coefficient vectors mirrors the
recursive structure in the definition of the $f_k$ by \eqref{e.bks},
we can write
\begin{equation}
  f_k = \sum_{i=1}^N \Rho_k^i \, \F_k^i 
  \equiv  \langle \Rho_k, \F_k \rangle \,.
\end{equation}
Likewise, the remainder \eqref{e.remainder} takes the form
\begin{equation}
  R_{n+1} 
  = J \sum_{k=n+1}^{2(n+1)} \eps^{k+2} \,
    \langle \Rho_{k+1}, \F_{k+1} \rangle \,.
  \label{e.remainder1}
\end{equation}

We first consider the family of coefficient vectors $\Rho_k$.  For a
general $\Rho \equiv (\Rho^1, \dots, \Rho^N)$, we define the norm
\begin{equation}
  \lvert \Rho \rvert
  = \max_{i=1,\dots,N} \, \lvert \Rho^i \rvert \,.
  \label{e.rhonorm}
\end{equation}
We then have the following estimate with respect to this norm.

\begin{lemma} \label{EstimRhoDeriv}
Let $\rho \in G^2(0,1)$ with $C=1$ and $\beta \leq 1$ in
\eqref{e.gevrey}.  Then
\begin{equation}
  \label{e.CorEstimGmak}
  \lvert \Rho_k \rvert \leq \frac{(k+1)!^2}{\beta^{k+1}} \,.
\end{equation}
\end{lemma}

\begin{proof}
We recursively define a family of function classes via
$\Gamma_1 = \{ \rho \}$ and $r \in \Gamma_k$ for $k \geq 2$
if there exists a nonnegative integer $m \in \N$, a
multi-index of length $s \in \N^*$ of strictly positive integers
$\alpha \in (\N^*)^s$, and functions $r_j \in \Gamma_{\alpha_j}$
for $j=1, \dots, s$ such that $k = m + \lvert \alpha \rvert$ and
\begin{equation}
  r = \partial_\tau^m \prod_{j=1}^s r_j \,.
  \label{e.decomposition}
\end{equation}
We note that the components of $\Rho_{k-1}$ are of class
$\Gamma_k$.  We shall show that $r \in \Gamma_k$ satisfies
\begin{equation}
  \sup_{\theta \in (0,1)} \lvert r(\theta) \rvert 
  \leq \frac{k!^2}{\beta^k} \,.
  \label{e.r-estimate}
\end{equation}
Due to the definition of the norm \eqref{e.rhonorm}, this implies
\eqref{e.CorEstimGmak}.

To prove \eqref{e.r-estimate}, we proceed by induction on $k$.  For
$k=1$, the statement is obvious.  Now suppose $k \geq 2$, so that $r$
has a decomposition of the form \eqref{e.decomposition}.  We can also
assume, without loss of generality, that when $s=1$,
$\lvert \alpha \rvert = \alpha_1 = 1$ and $m=k-1$.  In this case, the
statement is a direct consequence of the Gevrey class property
\eqref{e.gevrey}.  Now suppose that $s \geq 2$.  Then, by the Leibniz
rule,
\begin{equation}
  \lvert r(\theta) \rvert
  = \Biggl| 
      \sum_{\lvert \beta \rvert = m} \binom{m}{\beta}
      \prod_{j=1}^s \partial_\tau^{\beta_j} r_j (\theta)
    \Biggr|
  \leq \sum_{\lvert \beta \rvert = m} \binom{m}{\beta}
         \frac{(\alpha+\beta)!^2}{\beta^{\lvert \alpha + \beta
         \rvert}}
  \leq \frac{m!}{\beta^k} \frac{k!^2}{m!} \,,
\end{equation}
where the first inequality uses the induction hypothesis and the
second inequality is based on the observation that
$\lvert \alpha + \beta \rvert = m+ \lvert \alpha \rvert = k$ and a
combinatorial inequality which is stated and proved as
Lemma~\ref{l.combinatorial2} in the Appendix.
\end{proof}

We now turn to the family $\F_k$.  We define the corresponding
norms as follows.  For $z_0 \in \R^{2d}$ fixed and arbitrary $r>0$,
let $B_r(z_0)$ denote the closed ball of radius $r$ centered at $z_0$.
For a vector field $f$ on $\R^{2d}$, we write
\begin{equation}
  \lVert f \rVert_r
  = \sup_{z \in B_r(z_0)} \lVert f(z) \rVert
\end{equation}
and define a norm for $\F \equiv (\F^1, \dots, \F^N)$ via
\begin{equation}
  \lVert \F \rVert_r
  = \sum_{i = 1}^N \, \lVert \F^i \rVert_r \,.
  \label{e.f-norm}
\end{equation}
We now prove a variant of Cauchy's estimate in this setting.

\begin{lemma} \label{l.estimAnalitic}
Let $r>s>0$ and suppose the components of $\F$ and $\G$ are analytic
on $B_r(z_0)$.  Then
\begin{equation}
  \label{estimAnalitic}
  \lVert \D\F \otimes \G \rVert_s 
  \leq \frac{1}{r-s} \, \lVert \F \rVert_r \, \lVert \G \rVert_s \,.
\end{equation}
\end{lemma}

\begin{proof}
Let $h$ be any component of $\D\F \otimes \G$, i.e., there are
components $f,g$ of $\F$ and $\G$, respectively, such that
$h = \D f \, g$.  For fixed $z \in B_s(z_0)$, the function
$\phi(t) = f(z+t \, g(z))$ is analytic for
$\lvert t \rvert \leq \delta \equiv (r-s)/ \lVert g \rVert_s$.  Since
$\phi'(0) = h$, the classical Cauchy estimate implies
\begin{equation}
  \lVert h(z) \rVert
  = \lVert \phi'(0) \rVert
  \leq \frac1\delta \, 
       \sup_{\vert t \rvert \leq \delta} \lVert \phi(t) \rVert
  \leq \frac1{r-s} \, \lVert f \rVert_r \, \lVert g \rVert_s \,.
\end{equation}
Estimate \eqref{estimAnalitic} then follows from the definition of the
norm \eqref{e.f-norm}.
\end{proof}

\begin{lemma} \label{l.Fk_estimate} 
Let $z_0 \in \R^{2d}$, $R>0$, and $V$ be analytic on $B_{R(z_0)}$.
Then there exist constants $C>0$ and $\gamma>0$ such that for any
$n \in \N^\ast$ and $k \in \{ 0, \dots, 2n+3 \}$,
\begin{equation}
  \lVert \F_k \rVert_{R/2}
  \leq C \, \biggl( \frac{n}\gamma \biggr)^k \,.
  \label{e.Fk_estimate}
\end{equation} 
\end{lemma}

\begin{proof}
We set 
\begin{gather}
  \delta = \frac{R}{4n+6}
  \quad \text{and} \quad
  M = \max \biggl\{ 
             \sup_{z \in B_R(z_0)} \lvert \nabla V(z) \rvert,
             \delta
           \biggr\} \,,
           \label{eq:M}
\end{gather}
and recursively define the sequence $(S_k)$ via
\begin{equation}
  S_0 = 1 \,, 
  \qquad 
  S_{k+1} = S_k +\sum_{i+j = k} S_i \, S_j
\end{equation}
which has the asymptotic behavior
\cite[pp.~474--475]{flajolet2009analytic}
\begin{equation}
  \label{asymBehavior}
  S_k \sim \frac{(3-2\sqrt{2})^{-k-\frac 12}}{2\sqrt{\pi k^3}} \,.
\end{equation}
We will proceed to show that
\begin{equation}
  \lVert \F_k \rVert_{R - \delta k}
  \leq \frac{M^{k+1}}{\delta^k} \, S_k \,.
  \label{e.Fk_recursive}
\end{equation} 
The claimed estimate \eqref{e.Fk_estimate} is then a direct
consequence of \eqref{e.Fk_recursive}, \eqref{asymBehavior},
\eqref{eq:M} and $\delta k \leq R/2$.  Indeed, when $M=\delta$ in
\eqref{eq:M}, then \eqref{e.Fk_estimate} holds with
$\gamma = 3-2\sqrt{2}$.  Otherwise,
$M = \sup_{z \in B_R(z_0)} \lvert \nabla V(z) \rvert$, so that
choosing $\gamma = R \, (3-2\sqrt{2})/(10 \,M)$ will suffice.  The
minimum of both provides an $n$-independent choice of $\gamma$, with
similar considerations for $C$.

To prove \eqref{e.Fk_recursive}, we proceed by recursion on $k$.  For
$k=0$, the statement is trivial.  Now suppose the result is proved up
to index $k$.  Then, by Lemma~\ref{l.estimAnalitic},
\begin{align}
  \lVert \F_k \rVert_{R-\delta (k+1)}
  & \leq \lVert \F_k \rVert_{R-\delta (k+1)}
         + \frac1\delta \sum_{i+j=k} 
           \lVert \F_i \rVert_{R-\delta k} \, 
           \lVert \F_j \lVert_{R - \delta (k+1)}
    \notag \\
  & \leq \lVert \F_k \rVert_{R-\delta k}
         + \frac1\delta \sum_{i+j=k} 
           \lVert \F_i \rVert_{R-\delta i} \, 
           \lVert \F_j \lVert_{R - \delta j}
    \notag \\
  & \leq \frac{M^{k+2}}{\delta^{k+1}} \, 
         \biggl( 
           \frac{\delta}{M} \, S_k 
           + \sum_{i+j =k} S_i \, S_j
         \biggr) 
    \notag \\
  & \leq \frac{M^{k+2}}{\delta^{k+1}} \, S_{k+1} \,,
\end{align}
where the second inequality is due to the nesting of the balls over
which the supremum is taken, the third inequality is due to the
recursion hypothesis, and the last inequality uses $M\geq \delta$.
\end{proof}

\begin{proof}[Proof of Theorem~\ref{mainTheorem2}]
Without loss of generality, we assume that $a=1$.  Recalling the
expression for the remainder in the form \eqref{e.remainder1}, 
noting that 
\begin{equation}
  \lVert \langle \Rho, \F \rangle \rVert_r
  \leq \lvert \Rho \rvert \, \lVert \F \rVert_r \,,
\end{equation}
and referring to Lemma~\ref{EstimRhoDeriv} and
Lemma~\ref{l.Fk_estimate}, we estimate
\begin{align}
  \lVert R_{n+1} \rVert_{R/2}
  & \leq \sum_{k=n+2}^{2n+3} \eps^{k+2} \,
         \lvert \Rho_{k} \rvert \, \lVert \F_{k} \rVert_{R/2}
    \notag \\
  & \leq C \sum_{k=n+2}^{2n+3} \eps^{k+2} \,
         \frac{(k+1)!^2}{\beta^{k+1}} \,
         \biggl( \frac{n}\gamma \biggr)^k
    \notag \\
  & \leq C_1 \, \eps^2 \sum_{k=n+2}^{2n+3} \eps^{k} \,
         \frac{n^{3k}}{\alpha^{k}}
    \notag \\
  & \leq C_1 \, \frac{\delta^{n+3}}{1-\delta} \,.
  \label{e.Rn1_estimate}
\end{align} 
The third step is based on Stirling's inequality in the form
$m! < \e^{m-1} \, m^{m+1/2}$ for every $m \geq 2$, the inequality
$k+1 \leq 2n+3 \leq 5n$, and the observation that factors growing
algebraically in $k$ can always be absorbed by lowering $\alpha$ and
adjusting the overall multiplicative constant.  In the final step in
\eqref{e.Rn1_estimate} we have estimated the sum by the corresponding
infinite geometric series under the assumption that $\delta \equiv
\eps \, n^3/\alpha < 1$ and $\eps \leq 1$.  Let us now choose
\begin{equation}
  n = \biggl\lfloor \biggl( 
        \frac{\alpha \delta}{\eps}
      \biggr)^{\tfrac13} \biggr\rfloor \,.
\end{equation}
Then 
\begin{equation}
  \lVert R_{n+1} \rVert_{R/2}
  \leq \frac{C_1}{1-\delta} \, 
       \delta^{(\alpha \delta/\eps)^{\frac13}}
  \leq C_2 \, \e^{-c \eps^{-\frac13}} \,,
\end{equation}
where, in the last inequality, we have fixed $\delta \in (0,1)$ so
that $c>0$.  Following now the same steps as in the proof of
Theorem~\ref{mainTheorem} and using assumption \eqref{CRho} at $t=0$,
we observe that $w(0) = 0$ so that there exists a constant $C_3(T)$
such that
\begin{equation}
  \sup_{t\leq T} \, \lVert p(t) - \eps \, F_n(q(t),t) \rVert
  \leq C_3(T) \, \e^{-c \eps^{-\frac13}} \,.
\end{equation}
Using assumption \eqref{CRho} now at $t=T$, we verify that
\begin{equation}
  G_n (q^\ast) = \eps \, F_n(q(T),T)
\end{equation}
and the required estimate follows. 
\end{proof}

\begin{remark}
While it is possible to find ramp functions in $G^s$ for any $s>1$
satisfying the exponential order condition \eqref{CRho}, the proof as
stated will only generalize to Gevrey classes $s> 2$, as the required
generalization of Lemma~\ref{l.1},
\begin{gather}
  \sum_{m=0}^n (m+\ell)!^{s-1} \, (n+k-m-\ell)!^{s-1}
  \leq (n+k)!^{s-1} 
\end{gather}
fails for $s<2$.  For $s>2$, the final estimate reads
\begin{gather}
  \lVert p(T) - G_n(q^\ast) \rVert 
  \leq d \, \e^{-c \eps^{-\frac{1}{s+1}}} \,.
\end{gather}
As this is weaker than \eqref{MainEstimExp}, this generalization is of
little interest, in particular since a suitable ramp function in $G^2$
is easily available; see \eqref{e.rhoexp} below.
\end{remark}

As in Section~\ref{s.averaging}, we can combine our result with the
long-time invariance of approximate balance provided by
\cite[Theorem~2.1]{CotterReich06} as follows.

\begin{corollary}
In the setting of Theorem~\ref{mainTheorem2}, suppose $V$ is analytic
and strictly convex.  Then the method of optimal balance generates a
state which remains balanced to $O(\exp(-c/\eps^{1/3}))$ over times of
$O(\exp(c/\eps))$ under the dynamics of system \eqref{e.hamiltonian}.
\end{corollary}

\section{Numerical Tests}
\label{s.numerics}

A direct numerical demonstration of Theorems~\ref{mainTheorem}
and~\ref{mainTheorem2} is impossible as we do not have direct access
to the reference slow vector field $G_n$.  We
thus resort to computing the following proxy for the balance error. 
\begin{enumerate}
\item Given $q^*$, compute the corresponding $p^*$ via optimal
balance.  
\item Evolve the full system \eqref{e.hamiltonian}, initialized with
$q(0)=q^*$ and $p(0)=p^*$, forward up some time $t_1$ which is fixed
independent of $\eps$ on the slow time scale.  (For the simulations
shown below, $\eps t_1=0.5$).

\item ``Rebalance'' the evolved state, i.e., find a $p_1^*$ via
optimal balance for the given $q_1^*=q(t_1)$.

\item Define the diagnosed imbalance as
$I=\eps^{-1} \, \lVert p(t_1)-p_1^* \rVert$.
\end{enumerate}

We note that the diagnosed imbalance $I$ is not a direct measure of
the imbalance error $\lVert p(T) - G_n(q^\ast) \rVert$.  On the one
hand, $I$ may be overestimating imbalance because during the forward
simulation of model \eqref{e.hamiltonian}, there is a slow drift off
the slow manifold.  However, since $t_1$ is taken to be small, this
contribution is small as well as asymptotically subdominant.  A more
serious question is whether $I$ may underestimate the imbalance
because re-balancing may simply reproduce the same bias committed
during the initial balancing.  Since imbalanced motion is oscillatory
on the fast time scale, we would expect that the diagnosed imbalance
oscillates on the fast time scale as a function of $t_1$, so that the
amplitude of this oscillation can be taken as a measure of imbalance.
However, we did a careful pre-study which showed that $I$ depends
almost monotonically on $t_1$.  Thus, simply looking at the imbalance
for fixed $t_1$ already gives robust results.  Moreover, as we shall
see, the diagnosed imbalance reproduces the predictions of
Theorem~\ref{mainTheorem} accurately.  This gives strong empirical
support to the idea that $I$ is a useful proxy for imbalance which
could also be used in more complex situations, e.g., for the study of
rotating fluids.

\begin{figure}[tb]
\centering
\includegraphics{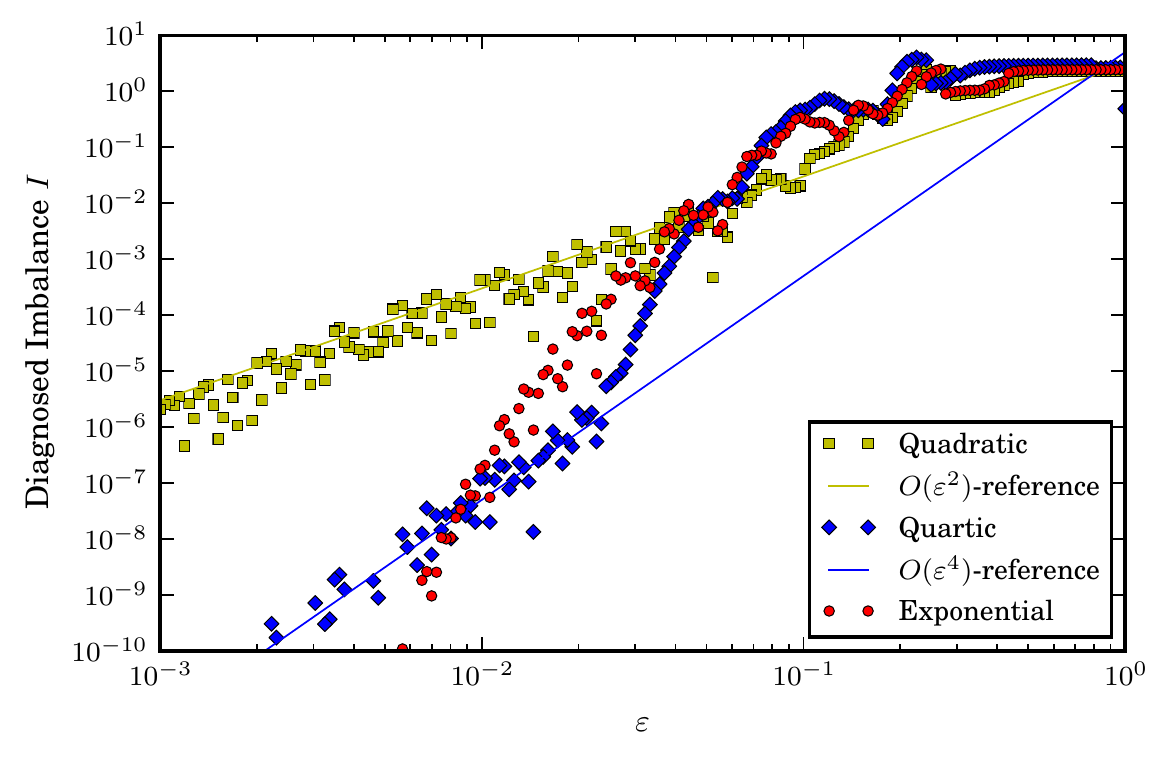}
\caption{Diagnosed imbalance $I$ as a function of $\eps$ for different
ramp functions.  ``Quadratic,''  ``quartic,'' and ``exponential''
refer to the  ramp functions \eqref{e.rhopoly} with $f(\theta)=x^2$,
$f(\theta)=x^4$, and $f(\theta)=\exp(-1/\theta)$, respectively.  The
ramp time is $T=2/\eps$.}
\label{f.nonlinear-order}
\end{figure}

\begin{figure}[tb]
\centering
\includegraphics{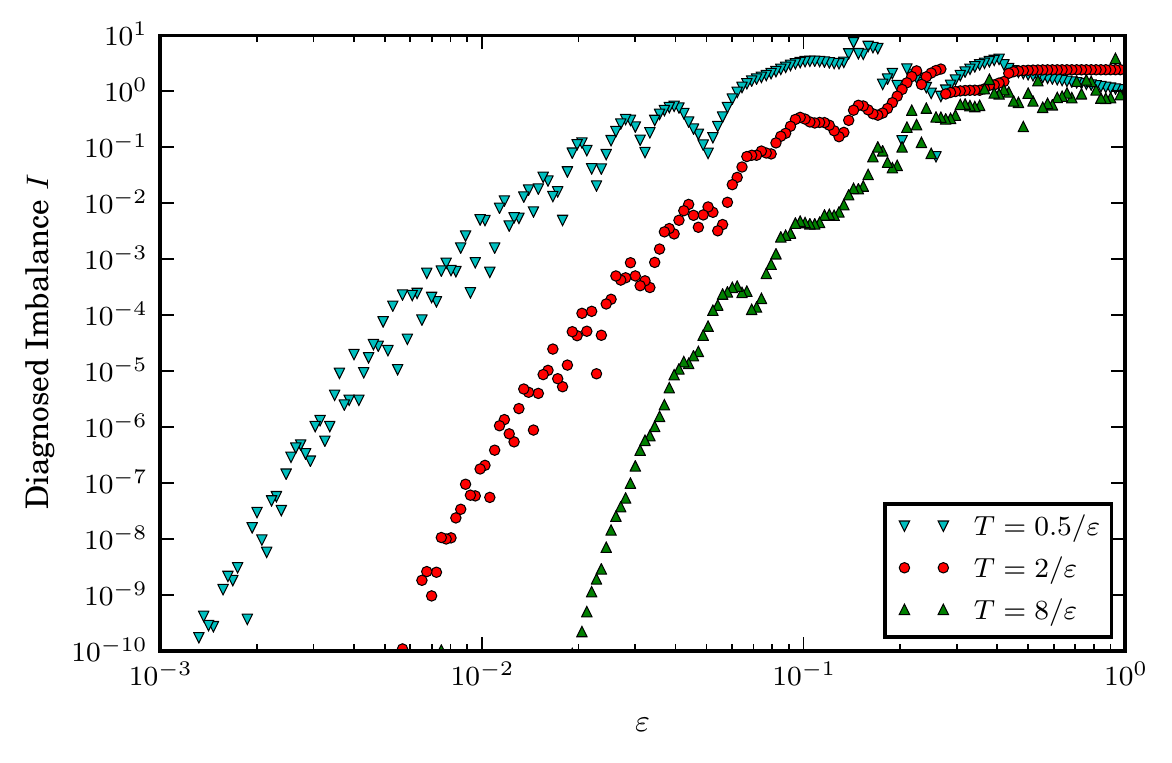}
\caption{Diagnosed imbalance $I$ as a function of $\eps$ for the when
taking the exponential ramp function \eqref{e.rhopoly} with
$f(\theta)=\exp(-1/\theta)$ for three different ramp times.}
\label{f.nonlinear-rampingtime}
\end{figure}

\begin{figure}[tb]
\centering
\includegraphics{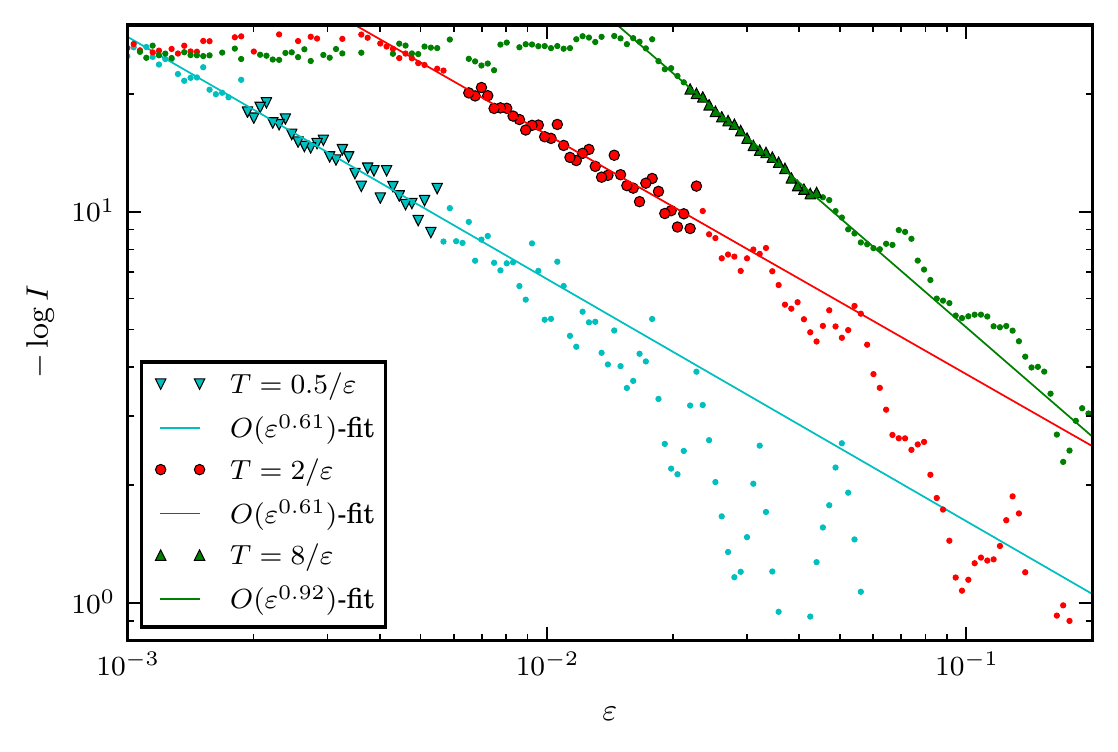}
\caption{The same data as Figure~\ref{f.nonlinear-rampingtime}, shown
on a doubly logarithmic vertical axis.  This allows a least square fit
to determine the power of $\eps$ in the exponent of the expression for
the exponential convergence rate, see equation \eqref{e.exprate}.  The
linear least square fit was performed over a finite interval in
$\eps$, indicated by the larger dots.}
\label{f.doublelog}
\end{figure}

In our proof-of-concept implementation, we use the potential
\begin{equation}
  V(q) = \tfrac34 \, q_1^4 + \tfrac14 \, q_2^4 \,.
\end{equation}
and solve the boundary value problem \eqref{e.ramp} by simple shooting
with an off-the-shelf ODE solver and root finder.  More efficient
implementations would use multiple shooting and possibly a symplectic
time-discretization.  The ramp functions are of the form
\begin{equation}
\label{e.rhopoly}
  \rho(\theta)
  = \frac{f(\theta)}{f(\theta) + f(1-\theta)}\, ,
\end{equation}
where
\begin{equation}
\label{e.rhoalg}
  f(\theta) = \theta^k
\end{equation}
for different exponents $k$, or $f(\theta) = \exp(-1/\theta)$ so that
\begin{equation}
  \label{e.rhoexp}
  \rho(\theta) 
  = \frac{e^{-1/\theta}}{e^{-1/\theta} + e^{-1/(1-\theta)}} \,.
\end{equation}
The ramp function \eqref{e.rhoexp} satisfies the exponential order
condition and is of Gevrey class $2$, thus it satisfies the
assumptions of Theorem~\ref{mainTheorem2}.\footnote{Indeed, each of
the terms appearing in \eqref{e.rhoexp} are of class $G^2$, see
Lemma~\ref{l.gevrey-function}.  As Gevrey classes are vector spaces,
the denominator of \eqref{e.rhoexp} is also of class $G^2$.  Finally,
nonsingular quotients of $G^2$-functions are again of class $G^2$, see
Lemma~\ref{l.gevrey-quotients}.}

In Figure~\ref{f.nonlinear-order}, we compare the performance of ramp
functions satisfying different order conditions.  For two algebraic
ramp functions with $k=2$ and $k=4$ in \eqref{e.rhoalg} corresponding
to $n=1$ and $n=3$ in the algebraic order condition of
Theorem~\ref{mainTheorem}, the predicted respective quadratic and
quartic decay of imbalance is clearly visible.  The super-algebraic
decay of imbalance for the ramp function with exponential order
condition is seen as a convex-shaped curve in the log-log plot of $I$
vs.\ $\eps$.

In Figure~\ref{f.nonlinear-rampingtime}, we explore the dependence of
the diagnosed imbalance $I$ on the ramp time $T$ for the exponential
ramp function case.  For a given value of $\eps$, longer ramp times
yield smaller diagnosed imbalances.  A rigorous study goes beyond the
Theorems proved here.

Figure~\ref{f.doublelog} shows the same data as
Figure~\ref{f.nonlinear-rampingtime}, but with a doubly logarithmic
vertical axis.  Assuming that the diagnosed imbalance behaves in the
general form suggested by Theorem~\ref{mainTheorem2}, i.e., if
\begin{equation}
  I = d \, \e^{-c \eps^{-\alpha}} \,,
  \label{e.imbalance-scaling}
\end{equation}
then
\begin{equation}
  \ln (\ln {d} - \ln{I}) =  \ln c - \alpha \, \ln \eps \,.
  \label{e.exprate}
\end{equation}
Then, plotting $\ln (-\ln I)$ vs.\ $\ln \eps$ should asymptote to a
line of slope $-\alpha$.  The data points show such behavior for a
good range of small values of $\eps$ small before the accuracy of the
time integrator and root solver, controlled to be at least $10^{-10}$,
breaks down.  The observed behavior is better than $\alpha=1/3$
obtained in the bounds of Theorem~\ref{mainTheorem2}, but depends on
the ramp time.  For large ramp times, the error is dominated by the
derivatives of the potential $V$ and the estimated exponent comes
close to $\alpha=-1$ that would be expected from the usual exponential
asymptotics \cite{CotterReich06}.  For shorter ramp times, the
influence of the ramp function becomes more important and the exponent
decreases, but appears to remain better than the theoretical bounds.

\section{Discussion}
\label{s.summary}

Our results show, in the context of a simple finite dimensional
Hamiltonian model problem, that the method of optimal balance yields a
state which is exponentially close to a balanced state obtained by
optimal truncation of an asymptotic series describing the approximate
slow manifold.  We believe that similar results will apply to more
general Hamiltonian fast-slow systems in the absence of resonances.

The result gives a strong support to the notion that optimal balance
may in fact be the best practically available characterization of a
slow manifold in this context.  As optimal truncation of an asymptotic
series is not computationally feasible, optimal balance could
therefore be used as a computable definition of a balanced state (this
idea has in fact been proposed earlier by McIntyre \cite{McIntyre09}).

However, a number of questions remain open.  An obvious question is
the sharpness of the analysis, both in terms of the current
restriction to ramp functions in Gevrey classes $G^s$ for $s \geq 2$,
and in terms of the exponent $\alpha$ in the imbalance scaling
\eqref{e.imbalance-scaling}.  A more practical concern is the best
choice of ramp time $T$ for fixed $\eps$.  Our analysis concerns only
the scaling with respect to $\eps$, but the structure of the estimates
as well as the numerical results suggest that at least initially the
results improve when the ramp time is increased.  This, however,
cannot go on forever because beyond some ramp time $T_{\text{opt}}$,
the imbalance due to the drift off the approximate manifold will
dominate and imbalance will increase as $T$ is increased further.  How
to design an adaptive algorithm which chooses an optimal ramp time
automatically is entirely open.

Whereas optimal balance has been successfully used in geophysical
fluid equations, the theory presented here was only developed for
finite-dimensional Hamiltonian systems. It is therefore a natural
question how our results translate to infinite-dimensional Hamiltonian
systems. A direct generalization of the model \eqref{e.hamiltonian} is
the semilinear Klein--Gordon equation in the non-relativistic limit
(e.g. \cite{Tsu_NRLKG}).  In general, obtaining results on approximate
slow manifolds for infinite dimensional Hamiltonian systems is
difficult since unbounded operators may destroy the scale separation
and the associated emergence of slow-fast or fast-fast resonances.
Existing results either apply to special solutions (e.g. \cite{Lu14}),
bounded slow subsystems (e.g. \cite{KristiansenWulff16}), or require
spatial analyticity of solutions (e.g. \cite{MatthiesScheel03}).
Finding the right analytical setting for the semi-linear Klein--Gordon
equation is a subject of ongoing research.

The question of justification of optimal balance in geophysical flow
problems is even more difficult, although our main motivation and
reported successful implementations come from this area.  Short of
rigorous justification, the issue of efficient implementation, in
terms of run-time and in terms of coding effort, is of considerable
practical relevance.  For the toy model considered here, we were able
to solve the optimal balance system problem \eqref{e.ramp} by simple
shooting.  However, this might fail or become excessively expensive in
higher dimensions.

Sophisticated boundary value solvers may be needed but are hard to
implement and computationally costly.  We remark that we have only
provided an approximate iterative solution of the
boundary-value-problem \eqref{e.ramp} and the issue of well-posedness
of the original boundary-value problem was not addressed.  Vi{\'u}dez
and Dritschel \cite{ViudezDritschel04} suggest an iterative procedure
where one integrates back and forth, resetting to the correct boundary
condition at each end. Empirically, their approach converges well in
the geophysical fluid dynamics context of their study. The iterative
back-and forth integrations can be understood as nudging towards the
boundary-values.  For linear systems, back-and-forth nudging can be
rigorously proven to converge to the true solution
\cite{AurouxB:2005:BackFN}; the problem considered here is, in our
understanding, not directly covered by these results but we expect
that a proof could be obtained with reasonable effort.  In our
concrete simulations, shooting was slightly more efficient than
back-and-forth nudging and converged for a moderately larger set of
parameters.  Finding the best strategy is an open problem.

\appendix

\section{Combinatorial estimates}

In the following, we prove an estimate on the combinatorial constants
which appear in the proof of Theorem~\ref{mainTheorem2}.  This result
is stated as Lemma~\ref{l.combinatorial2} below.  We begin with a
special case which is needed in the proof of the general result.
 
\begin{lemma} \label{l.1}
Let $n \in \N$ and $k, \ell \in \N^*$ with $1 \leq \ell < k$.  Then
\begin{gather}
  \sum_{m=0}^n (m+\ell)! \, (n+k-m-\ell)!
  \leq (n+k)! \,.
\end{gather}
\end{lemma}

\begin{proof}
We proceed by induction on $n$.  For $n=0$, the statement is obvious.
Now suppose the statement is true up to step $n-1$.  Then
\begin{align}
  \sum_{m=0}^n (\ell+m)! \, & (n+k-\ell-m)!
    = (\ell+n)! \, (k-\ell)!
      + \sum_{m=0}^{n-1} (\ell+m)! \, (n+k-\ell-m)!
      \notag \\
  & \leq (n+k-1)! 
      + (n+k-1) \sum_{m=0}^{n-1} (\ell+m)! \, (n-1+k-\ell-m)!
      \notag \\
  & \leq (n+k-1)! 
      + (n+k-1) \, (n-1+k)!
    = (n+k)! 
\end{align}
where the first inequality is due to $1 \leq \ell < k$ and the second
inequality uses the induction hypothesis.
\end{proof}

\begin{lemma} \label{l.combinatorial2}
Let $n \in \N$, and $s \in \N^*$.  Then for a multi-index of strictly
positive integers $\alpha \in (\N^*)^s$ with $\lvert \alpha \rvert = k$, 
\begin{gather}
  \sum_{\lvert \beta \rvert = n} 
    \frac{(\alpha + \beta)!^2}{\beta!}
  \leq \frac{(n+k)!^2}{n!} \,,
  \label{e.l2}
\end{gather}
where the sum is over multi-indices $\beta$ of length $s$.
\end{lemma}

\begin{proof}
We proceed by induction on $s$.  For $s=1$, the two sides of
\eqref{e.l2} are trivially equal.  Now suppose the statement holds
true up to step $s-1$.  We write $\alpha = (\alpha', \ell)$ and
$\beta = (\beta', m)$, where $\alpha'$ and $\beta'$ are multi-indices
of length $s-1$, and $\ell$ and $k$ are integers satisfying
$1 \leq \ell < k$.  Then
\begin{align}
  \sum_{\lvert \beta \rvert = n} 
    \frac{(\alpha + \beta)!^2}{\beta!}
  & = \sum_{m=0}^n \frac{(\ell+m)!^2}{m!}
        \sum_{\lvert \beta' \rvert = n-m} 
        \frac{(\alpha' + \beta')!^2}{\beta'!}
    \notag \\
  & \leq \sum_{m=0}^n \frac{(\ell+m)!^2}{m!} \,
        \frac{(n-m+k-\ell)!^2}{(n-m)!} 
    \notag \\
  & = \frac{(n+k)!}{n!} \sum_{m=0}^n
        \binom{n}{m} \, \binom{n+k}{m+\ell}^{-1} \,
        (m+\ell)! \, (n+k-m-\ell)! \,.
\end{align}
As the ratio of the binomial coefficients that appear in the right
hand sum is always bounded above by $1$, the proof is achieved by
Lemma~\ref{l.1}.
\end{proof}

\section{$G^2$-estimates on the exponential ramp function}

The following two results are necessary to show that the exponential
ramp function \eqref{e.rhoexp} used above in the numerical experiments
is of Gevrey class $2$.  We believe that the results are classical;
Lemma~\ref{l.gevrey-function}, for example, is stated without proof in
\cite[p.\ 218]{ito1993encyclopedic}.  In this appendix, we give
complete proofs for the convenience of the reader.

\begin{lemma} \label{l.gevrey-function}
The function
\begin{gather}
  f(x) =
  \begin{cases}
    0 & \text{for } x \leq 0 \\
    \exp(-1/x) & \text{for } x>0 
  \end{cases}
\end{gather}
is of Gevrey class $2$ uniformly in $\R$.
\end{lemma}

\begin{proof}
The function $f$ is holomorphic in the right complex half-plane.
Fixing $\lambda \in (0,\tfrac12)$, the Cauchy integral formula 
\begin{equation}
  f^{(n)}(x) 
  = \frac{n!}{2\pi i}
    \int_{\gamma}\frac{f(z)}{(z-x)^{n+1}} \, \d z
\end{equation}
applies in particular when $\gamma$ is a circle of radius $\lambda x$
centered at $x$.  We parameterize $\gamma$ writing
$z(\theta) = x + \lambda x \, w(\theta)$ where $w(\theta)$ is an
arc-length parameterization of the unit circle.  Then
\begin{align}
  \lvert f^{(n)}(x) \rvert
  & \leq \frac{n!}{2\pi \, (\lambda x)^{n+1}}
        \int_0^{2\pi} \lvert f(z(\theta)) \rvert \, \d \theta
    \notag \\
  & \leq \frac{n!}{(\lambda x)^{n+1}}
       \sup_{\theta \in [0, 2\pi]}
       \biggl|
         \exp \biggl(
                - \frac{1+\lambda \, \overline w(\theta)}%
                       {x \, \lvert 1 + \lambda \, w(\theta) \rvert^2}
              \biggr)
       \biggr|
    \notag \\
  & 
  \leq \frac{n!}{(\lambda x)^{n+1}}
       \exp \biggl(
              - \frac{1-\lambda}%
                     {x \, \lvert 1 + \lambda \rvert^2}
            \biggr) \,.       
\end{align}
Maximizing the right hand side with respect to $x$ and using
Sterling's inequality in the form $m^m \, \e^{-m} \leq m!$, we obtain
\begin{equation}
  \sup_{x \in \R} \lvert f^{(n)}(x) \rvert
  \leq \frac{(n+1)!^{2}}{\eta^{n+1}}
  \label{e.gevrey2estimate}
\end{equation}
with $\eta = \lambda(1 - \lambda)/(1+\lambda)^2$.  This proves that
$f$ is of Gevrey class $2$.
\end{proof}

\begin{lemma} \label{l.gevrey-quotients}
Let $U \subset \R$ be open and suppose $f,g \in G^2(U)$ with
$g \geq c > 0$ for some constant $c$.  Then $h=f/g \in G^2(U)$.
\end{lemma}

\begin{proof}
Without loss of generality, assume that $f\leq 1$ and $g \geq 1$ on
$U$, so that $h \leq 1$.  Further, let $\alpha$ denote the smaller of
the two parameters appearing in the denominator of the Gevrey class
estimates \eqref{e.gevrey} of $f$ and $g$.  Set $\beta = \alpha/3$.
Using the Leibniz rule for the $n$th derivative of the product $gh$
and rearranging terms, we have
\begin{equation}
  h^{(n)} 
  = \frac1g \, \biggl(
                 f^{(n)} - \sum_{j=0}^{n-1} \binom{n}{j} \, 
                 g^{(n-j)} \, h^{(j)}
               \biggr) \,.
\end{equation}
We now proceed by induction on $n$.  For $n=0$, the statement is
obvious.  Now suppose that $h$ satisfies a Gevrey class estimate of
the form \eqref{e.gevrey} with parameter $\beta$ up to order $n-1$.
Then
\begin{align}
  \lvert h^{(n)} (x) \rvert
  & \leq \frac{n!^2}{\alpha^n}
         + n! \sum_{j=0}^{n-1} \frac{(n-j)!}{\alpha^{n-j}} \, 
              \frac{j!}{\beta^{j}}
    \leq \frac{n!^2}{\alpha^n} \, 
         \biggl(
           1 + \frac{3^{n-1}}{n!} 
           \sum_{j=0}^{n-1} (n-j)! \, j!
         \biggr)
    \leq \frac{n!^2}{\beta^n} \,,
\end{align}
where the last inequality is based on the observation that
\begin{gather}
  \sum_{j=0}^{n-1} (n-j)! \, j! 
  = n! + \sum_{j=1}^{n-1} (n-j)! \, j! 
  \leq n! + (n-1) \, (n-1)!
  \leq 2 \, n!
\end{gather}
and further that $1+2 \cdot 3^{n-1} \leq 3^n$.
\end{proof}

\section*{Acknowledgments}

We thank David Dritschel for many useful discussions on balance and
optimal potential vorticity balance.  The numerical study is based on
prior work done by Zekun Yang as part of her Bachelor thesis at Jacobs
University.  This paper contributes to the project ``The interior
energy pathway: internal wave emission by quasi-balanced flows'' of
the Collaborative Research Center TRR 181 ``Energy Transfers in
Atmosphere and Ocean'' funded by the German Research Foundation.
Funding through the TRR 181 is gratefully acknowledged.  HM and MO
further acknowledge funding by German Research Foundation grant
OL-155/6-1.

%\bibliographystyle{siam}
%\bibliography{bibliography_balance}

\begin{thebibliography}{10}

\bibitem{AurouxB:2005:BackFN}
{\scshape D.~Auroux and J.~Blum}, {\em Back and forth nudging algorithm for
  data assimilation problems}, C. R. Acad. Sci. Paris, Ser. I, 340 (2005),
  pp.~873--878.

\bibitem{Cotter13}
{\scshape C.~Cotter}, {\em {Data assimilation on the exponentially accurate
  slow manifold}}, Phil. Trans. R. Soc. A, 371 (2013), p.~20120300.

\bibitem{CotterReich06}
{\scshape C.~J. Cotter and S.~Reich}, {\em Semigeostrophic particle motion and
  exponentially accurate normal forms}, Multiscale Model. Simul., 5 (2006),
  pp.~476--496 (electronic).

\bibitem{cullen2008comparison}
{\scshape M.~J.~P. Cullen}, {\em A comparison of numerical solutions to the
  eady frontogenesis problem}, Quart. J. R. Meteorol. Soc., 134 (2008),
  pp.~2143--2155.

\bibitem{flajolet2009analytic}
{\scshape P.~Flajolet and R.~Sedgewick}, {\em Analytic combinatorics},
  Cambridge University Press, 2009.

\bibitem{FrankGottwald13}
{\scshape J.~E. Frank and G.~A. Gottwald}, {\em Stochastic homogenization for
  an energy conserving multi-scale toy model of the atmosphere}, Phys. D, 254
  (2013), pp.~46--56.

\bibitem{Gottwald14}
{\scshape G.~A. Gottwald}, {\em {Controlling balance in an ensemble {K}alman
  filter}}, Nonlinear Proc. Geoph., 21 (2014), pp.~417--426.

\bibitem{GottwaldO:2014:SlowDD}
{\scshape G.~A. Gottwald and M.~Oliver}, {\em Slow dynamics via degenerate
  variational asymptotics}, Proc. R. Soc. Lond. Ser. A Math. Phys. Eng. Sci.,
  470 (2014), p.~20140460.

\bibitem{ito1993encyclopedic}
{\scshape K.~It{\^o}}, {\em Encyclopedic Dictionary of Mathematics}, vol.~1,
  MIT Press, 1993.

\bibitem{KristiansenWulff16}
{\scshape K.~U. Kristiansen and C.~Wulff}, {\em Exponential estimates of
  symplectic slow manifolds}, J. Differential Equations, 261 (2016),
  pp.~56--101.

\bibitem{LeungMeyer75}
{\scshape A.~Leung and K.~Meyer}, {\em Adiabatic invariants for linear
  {H}amiltonian systems}, J. Differential Equations, 17 (1975), pp.~32--43.

\bibitem{Lu14}
{\scshape N.~Lu}, {\em Small generalized breathers with exponentially small
  tails for {K}lein-{G}ordon equations}, J. Differential Equations, 256 (2014),
  pp.~745--770.

\bibitem{Lynch}
{\scshape P.~Lynch}, {\em {The Emergence of Numerical Weather Prediction:
  Richardson's Dream}}, Cambridge University Press, Cambridge, 2006.

\bibitem{LynchHuang92}
{\scshape P.~Lynch and X.-Y. Huang}, {\em Initialization of the {HIRLAM} model
  using a digital filter}, Mon. Weather Rev., 120 (1992), pp.~1019--1034.

\bibitem{MacKay04}
{\scshape R.~S. MacKay}, {\em {Slow manifolds}}, in Energy Localisation and
  Transfer, T.~Dauxois, A.~Litvak-Hinenzon, R.~S. MacKay, and A.~Spanoudaki,
  eds., World Scientific, 2004, pp.~149--192.

\bibitem{MatthiesScheel03}
{\scshape K.~Matthies and A.~Scheel}, {\em Exponential averaging for
  {H}amiltonian evolution equations}, Trans. Amer. Math. Soc., 355 (2003),
  pp.~747--773.

\bibitem{McIntyre09}
{\scshape M.~E. McIntyre}, {\em Spontaneous imbalance and hybrid vortex-gravity
  structures}, J. Atmos. Sci., 66 (2009), pp.~1315--1326.

\bibitem{McI15}
{\scshape M.~E. McIntyre}, {\em Balanced flow}, in Encyclopedia of Atmospheric
  Sciences, J.~Pyle and F.~Zhang, eds., Academic Press, Oxford, second~ed.,
  2015, pp.~298--303.

\bibitem{Neef06}
{\scshape L.~J. Neef, S.~M. Polavarapu, and T.~G. Shepherd}, {\em
  Four-dimensional data assimilation and balanced dynamics}, J. Atmos. Sci., 63
  (2006), pp.~1840--1858.

\bibitem{Neishtadt81}
{\scshape A.~I. Ne{\u\i}shtadt}, {\em On the accuracy of conservation of the
  adiabatic invariant}, J. Appl. Math. Mech., 45 (1981), pp.~58--63.

\bibitem{Neishtadt84}
\leavevmode\vrule height 2pt depth -1.6pt width 23pt, {\em The separation of
  motions in systems with rapidly rotating phase}, Prikl. Mat. Mekh., 48
  (1984), pp.~197--204.

\bibitem{Nekhoroshev77}
{\scshape N.~Nekhoroshev}, {\em An exponential estimate of the time of
  stability of a nearly-integrable {H}amiltonian system}, Russ. Math. Surv., 32
  (1977), pp.~1--65.

\bibitem{Oliver06}
{\scshape M.~Oliver}, {\em Variational asymptotics for rotating shallow water
  near geostrophy: a transformational approach}, J. Fluid Mech., 551 (2006),
  pp.~197--234.

\bibitem{Tsu_NRLKG}
{\scshape M.~Tsutsumi}, {\em Nonrelativistic approximation of nonlinear
  {K}lein-{G}ordon equations in two space dimensions}, Nonlinear Anal., 8
  (1984), pp.~637--643.

\bibitem{Van13}
{\scshape J.~Vanneste}, {\em Balance and spontaneous wave generation in
  geophysical flows}, Ann. Rev. Fluid Mech., 45 (2013), pp.~147--172.

\bibitem{ViudezDritschel04}
{\scshape {\'A}.~Vi{\'u}dez and D.~G. Dritschel}, {\em Optimal potential
  vorticity balance of geophysical flows}, J. Fluid Mech., 521 (2004),
  pp.~343--352.

\bibitem{VonStorchBO:2017:InteriorEP}
{\scshape J.-S. von Storch, G.~Badin, and M.~Oliver}, {\em The interior energy
  pathway: inertial gravity wave emission by oceanic flows}.
\newblock In preparation.

\end{thebibliography}

\end{document}